\documentclass[a4paper,10pt,intlimits,ngerman]{amsart}
\usepackage[varg]{txfonts}
\usepackage[unicode]{hyperref}
\usepackage{amsfonts}
\usepackage[mathscr]{euscript}
\usepackage{mathrsfs}
\usepackage{graphicx}
\usepackage{color}
\usepackage{tikz-cd}
\usetikzlibrary{calc,patterns,angles,quotes}

\newtheorem{theorem}{Theorem}[section]
\newtheorem{proposition}[theorem]{Proposition}
\newtheorem{lemma}[theorem]{Lemma}
\newtheorem{corollary}[theorem]{Corollary}

\newtheorem{remark}[theorem]{Remark}

\hypersetup{
breaklinks=true,
colorlinks=true,
linkcolor=blue,
citecolor=blue,
urlcolor=blue,
}

\numberwithin{equation}{section}

\title[]{On the analyticity of solutions to non-linear elliptic partial differential equations}
\author{
Simon Blatt}
\address[Simon Blatt]{Departement of Mathematics, Paris Lodron Universit\"at Salzburg, Hellbrunner Strasse 34, 5020 Salzburg, Austria}
\email[Simon Blatt]{simon.blatt@sbg.ac.at}

\date{\today}

\begin{document}

\begin{abstract}
We present an elementary and easy proof of the fact that $C^\infty$ solutions to non-linear elliptic equations of second order
$$
 \phi(x, u, D u, D^2 u)=0
$$
are analytic. Following ideas of Kato \cite{Kato1996}, the proof uses an inductive estimate for weighted derivatives. We conclude the proof using
Cauchy's method of majorants \cite{Cauchy1842}.
\end{abstract}

\maketitle

\tableofcontents

\section{Introduction}

%%
%
% What kind of equation are we looking at?
%
%%

We consider a solution $u \in C^\infty(\Omega, \mathbb R)$ to the fully non-linear equation
\begin{equation} \label{eq:PDE}
 \phi (x,u(x), D u(x), D^2 u(x)) = 0 \text { on  } \Omega.
\end{equation}
Here,  $\Omega \subset \mathbb R^n$ is an open subset of the Euclidean space of dimension $n \geq 2$. Furthermore, we assume that
$
 \phi: \Omega \times \mathbb R \times \mathbb R^n \times S(n) \rightarrow \mathbb R
$
is a real analytic function that is elliptic in the sense that 
the functions $a_{ij} (x,y,p,q):= \frac { \partial f}{ \partial q_{ii}} (x,y,p,q)$ satisfy
\begin{equation} \label{eq:ellipticity}
 0 < a_{ij} (x,y,p,q) \xi_i \xi_j  \quad \text{for all }\xi \in \mathbb R^n. 
\end{equation}

%%
%
% Statement of the theorem
%
%%

We will give a short and elementary proof of the following well-known result:

\begin{theorem} \label{thm:Analyticity}
Let $u\in C^\infty (\Omega, \mathbb R)$ be a solution to \eqref{eq:PDE}, $x_0 \in \Omega$, and $\phi$ be analytic in a neighborhood of $(x_0,u(x_0), Du(x_0), D^2(x_0)$. Then $u$ is real analytic near $x_0$.
\end{theorem}

%%
%
% Historic remarks - overview of the methods used!
%
%%

Shortly after Hilbert asked whether solution to regular variational problems are always analytic in his famous speech at the ICM in 1900 \cite[Problem 19]{Hilbert1902}, Bernstein  gave the first proof 
of the theorem above for the case $n=2$ \cite{Bernstein1904}  under the assumption that $u \in C^3$. Gevrey 
\cite{Gevrey1918} extend Bernstein's result to parabolic equations inventing a completely different method. Petrowsky \cite{Petrowsky1939} generalised the result for a function in two independent variables to systems on a Euclidean space of arbitrary finite dimension. The most general results for elliptic systems can be found in \cite{Friedman1958} and \cite{Morrey1958,Morrey1958a}

In this article we want to give a simple proof of analyticity of a solution $u$ as 
above. Our proof uses cut-off functions in a way inspired by the work of Kato 
\cite{Kato1996} to derive an recursive estimate for the weighted differentials of 
$u$. Unfortunately, Kato only considered the case $$ \Delta u = u^2$$ in his paper 
and announced to extend his proof to arbitrary non-linear equations.
In \cite{Hashimoto2006} Hashimoto tried to carry out this program. In this paper we clarify and complete the proof of Hashimoto.

Our proof combines basic $L^2$ estimates for the Laplacian with a higher order chain rule to derive in straight forward way a recursive estimate for weighted derivates. We use Cauchy method of majorants \cite{Cauchy1842} to show that these recursive estimates imply the analyticity of $u$. That means that we compare the quantities to the derivatives of a solution to a carefully chosen analytic ordinary differential equation. Using that such solutions are known to be analytic, we can close the argument.

We will collect some well know results in Section \ref{sec:prel} to make the article accessible to a broader audience. We characterise  analytic functions (Subsection \ref{sec:anal}), repeat basic $L^2$ estimates for solutions to  Poisson's equation (Subsection \ref{sec:l2}) and state a higher order chain rule (Subsection \ref{sec:bruno}). Furthermore, we prove some identities that help us to deal with the estimates we obtain applying these two ingredients to the problem. In Section \ref{sec:proof} we end this note with a complete proof of Theorem \ref{thm:Analyticity}.

\section{Preliminaries} \label{sec:prel}

%%
%
% Characterizations of analytic functions
%
%%

\subsection{Characterisation of analytic functions} \label{sec:anal}

Our proof relies on the following well known fact:

\begin{lemma} \label{lem:CharacterizationAnalyticity}
A function $u: \Omega \rightarrow \mathbb R^m$, $\Omega \subset \mathbb R^n$ is 
analytic if and only if for every compact set $K \subset \Omega$
there exist constants $C=C_K, A=A_K < \infty$ such that for every 
multi-index $\alpha \in \mathbb R^n$ we have
$$
 \|\partial^\alpha u\|_{L^\infty(K)} \leq C A^{|\alpha|} |\alpha|!.
$$
\end{lemma}

A proof of this statement can be found in \cite{Krantz1992}. Combining this with Sobolev's embedding theorems we immediately get

\begin{corollary} \label{cor:CharacterizationAnalyticity}
 A function $u: \Omega \rightarrow \mathbb R^m$, $\Omega \subset \mathbb R^n$ is 
analytic if and only if for every $x \in \Omega$ there is a radius $r>0$
and constants $C, A < \infty$ such that for every multi-index $\alpha \in 
\mathbb R^n$ we have
$$
 \|\partial^\alpha u\|_{L^2(B_r(x))} \leq C A^{|\alpha|} |\alpha|! 
$$
\end{corollary}

It is our aim in the proof of Theorem \ref{thm:Analyticity} in Section \ref{sec:proof}, to derive such estimates for a solution of \eqref{eq:PDE} around the point $x_0 \in \Omega$ and its derivatives.

%% 
%
% L^2 estimates on the whole space
%
%%

\subsection{\protect{$L^2$-estimates}} \label{sec:l2}

We recapitulate the following well-known $L^2$-estimate for the Laplacian. 

\begin{proposition} \label{prop:L2Estimates}
For $u \in C^2(\mathbb R^n)$ with compact support we have
$$
 \|D^2u\|_{L^2} = \|\Delta u\|_{L^2}.
$$
\end{proposition}

\begin{proof}
 Applying Green's first identity twice we get 
$$
 \int_{\mathbb R^n} \Delta u \Delta u dx = - \int_{\mathbb R^n} \Delta \partial_i u \partial_i u dx
 = \int_{\mathbb R^n } \partial_{ji} u \partial_{ji}u dx.
$$
\end{proof}
For $m= \lfloor \frac n 2 \rfloor +1$ we set
$$
 \|f\|_{H^m } = \sum_{|\alpha|=m} \|\partial^\alpha u\|_{L^2}.
$$
Then the following proposition holds:
\begin{proposition} \label{prop:HmEstimates}
For $u \in C^{2+m}(\mathbb R^n)$ with compact support we have
$$
 \|D^2u\|_{H^m} = \|\Delta u\|_{H^m}.
$$
\end{proposition}

\begin{proof}
From Proposition \ref{prop:L2Estimates} we get 
$$
   \| \partial^\alpha D^2 u \|_{L^2 } =  \| D^2 \partial^\alpha u \|_{L^2 }  = \| \Delta \partial^\alpha u \|_{L^2}
   =\| D^\partial \Delta u \|_{L^2}.
$$
Summing over all $\alpha \in \mathbb N^n$ with $|\alpha| = m$ yields the result.
\end{proof}

\subsection{A Banach algebra}

We will use later on that $H^m(\mathbb R^n)$ is a Banach algebra, if we chose$m=\lfloor \frac n2 \rfloor +1$:

\begin{proposition} [\protect{\cite{Adams1975}}]
For $f,g \in C^{\infty}(\mathbb R^n)$ with compact support and $m = \lfloor \frac n2 \rfloor +1$ we have
$$
 \|fg\|_{H^m} \leq C \| f\|_{H^m} \|g\|_{H^m}.
$$
\end{proposition}

\begin{remark} \label{rem:BanachAlgebra}
In order to get rid of the constant $C$ in the last inequality, we work with the norms
\begin{equation} \label{eq:BanachAlgebra}
\|f\|_{\tilde H^m} := C \|f\|_{ H^m}
\end{equation}
in the following where. Then 
$$
 \|fg\|_{\tilde H^m} =  C \|fg\|_{H^m } \leq C^2 \|f\|_{H^m} \|g\|_{H^m} \leq \|f\|_{\tilde H^m} \|g\|_{\tilde H^m}.
$$

\end{remark}

\subsection{Binomials}

For two multiindices $\alpha=(\alpha_1, \ldots, \alpha_n), \beta= (\beta_1, \ldots, \beta_n) \in \mathbb N_0^n$ the binomial coefficient is defined through
$$
 \binom{\alpha}{\beta} = \prod_{j=1}^n \binom{\alpha_j}{\beta_j}.
$$
We will use the following result to deal with terms coming from the generalised 
product rule. 

\begin{proposition}[\protect{\cite[Proposition 2.1]{Kato1996}}] \label{prop:Binomials}
 Let $\alpha$ be a multiindex and $k \leq |\alpha|$. Then 
 $$
 \sum_{\beta \leq \alpha, |\beta|=k} \binom{\alpha}{\beta} = \binom{|\alpha|}{k}.
 $$
 \end{proposition}
 
 \begin{proof}
  This can be seen comparing the coefficients in the Taylor series expansions of the 
identity
$$
 (1+t)^{\alpha_1} \cdots (1+t)^{\alpha_n} = (1+t)^{|\alpha|}.
$$
 \end{proof}

\subsection{Higher order chain rule} \label{sec:bruno}

The following fact for higher derivatives will be important in our proof:

\begin{proposition} \label{prop:FraaDiBruno}
 Let $g:\mathbb R^{m_1} \rightarrow \mathbb R^{m_2}$ and $f: \mathbb R^{m_2} \rightarrow \mathbb R$ be two $C^k$-functions. Then for any multiindex 
 $\alpha \in \mathbb N ^{m_1}$ of length $|\alpha| \leq k$ and $x \in \Omega$ the identity
  $$
  \partial^\alpha (f \circ g ) (x) = P_{m_1, m_2}^\alpha ( \{\partial^\gamma f(g(x))\}_{|\gamma| \leq |\alpha|}, \{\partial^\gamma g_i\}_{0 \leq \gamma \leq \alpha} )
 $$
holds,  where $P^\alpha_{m_1,m_2}$ is a fixed linear combination with positive coefficients of terms of the form 
 $$
  \partial^l_{x_{i_1}, x_{i_l}}( g(x))  \partial^{\gamma_1}g_{i_1} \cdots  \partial^{\gamma_l}g_{i_l}  
 $$
 with $1 \leq l \leq |\alpha|$ and $|\gamma_1| + \ldots + |\gamma_l| = |\alpha|.$.
 \end{proposition}
 
For $m_1=m_2=1$ we will use the notation $P^k$ instead of $P^\alpha_{m_1,m_2}.$ 
 We leave the easy inductive proof of this statement to the reader. This proposition contains all the information about the higher order chain rule we need. Precise and explicite formulas for the higher order chain rule were given by Faa di Bruno for the univariate  case \cite{DiBruno1857}  and by for example Constanini and Savits in \cite{Constantine1996} for the multivariate case.
 
 Let us derive two easy consequences of Proposition \ref{prop:FraaDiBruno}. The first one will allow us to reduce the multivariate case to the univariate one.
 
 \begin{lemma} \label{lem:SingleVariable}
 For constants $a_\gamma= a_{|\gamma|}, b_{\gamma} = b_{|\gamma||} \in 
\mathbb R $, depending only on the 
length of the multiindex $\gamma$, we have
$$
 P^\alpha_{m_1, m_2} (\{a_{|\gamma|}\},\{b_{|\gamma|}\}) = 
P^{|\alpha|}(\{a_{|\gamma|}\}, \{b_{|\gamma|}\}).
$$
 \end{lemma} 
 
 \begin{proof}
   Plugging functions $g$ and $f$ of the form $$g(x_1, \ldots ,x_{m_1}) = \tilde 
g( x_1+ \dots + x_{m_1} ) \cdot (1, \ldots, 1)^t$$ and $$f(y_1, \ldots 
y_{m_2})=\tilde f\left(\frac{y_1+ \dots + y_{m_2}}{m_2}\right)$$ into the higher 
order chain rule (Proposition \ref{prop:FraaDiBruno}), we get from 
$$
 f\circ g = (\tilde f \circ \tilde g) (x_1 + \dots + x_{m_1})
$$
that
$$
P^\alpha_{m_1, m_2} (\{\partial^\gamma f(g(x))\}_{|\gamma| \leq |\alpha|}, \{\partial^\gamma g_i\}_{0 \leq \gamma \leq \alpha}) = P^{|\alpha|} (\{\partial^l \tilde f( \tilde g(x))\}_{l \leq |\alpha|}, \{\partial^l \tilde g\}_{0 \leq l \leq |\alpha|}).
$$
So for constants $a_\gamma= a_{|\gamma|}, b_{\gamma} = b_{|\gamma|} \in 
\mathbb R $ depending only on the 
length of the multiindex $\gamma$, we obtain
$$
 P^\alpha_{m_1, m_2} (\{a_{|\gamma|}\},\{b_{|\gamma|}\}) = 
P^{|\alpha|}\{\{a_{|\gamma|} \}, \{b_{|\gamma|} \}\}.
$$

 \end{proof}

 We will use Lemma \ref{lem:SingleVariable} to estimate derivatives of analytic function in $x, u(x), Du(x),$ and $D^2u(x)$. Let $u \in C^\infty(B_1(0), \mathbb R)$ and $f: \Omega \times \mathbb R \times \mathbb R^n \times S(n) \rightarrow \mathbb R $ be a function that is infinitely often differentiable on a neighborhood of the image 
 $$K = \{(x,u(x),Du(x), D^2u(x): x \in B_1(0))\}.$$ 
 For a fixed cutoff function $\rho \in 
C^\infty ( \mathbb R^n,  [0,1])$ with support 
contained in $B_1(0)$ and $\rho=1$ on $B_{\frac 12 }(0)$  we will consider the 
quantities
\begin{equation*}
\begin{cases}
 \tilde M_N := \sup_{|\alpha |=N} \| \partial^\alpha u\|_{\tilde H^m(B_1(0))} &\text{ for } N = 
0,1,2\\
 \tilde M_N := \sup_{|\alpha |=N} \| \rho^{N-3} \partial^\alpha u\|_{\tilde H^m} \quad 
&\text{ for } N \geq 3
\end{cases}
\end{equation*}
and
\begin{align*}
 M_N := \tilde M_{N+2} + \tilde M_{N+1} + \tilde M_N +1
\end{align*}
for all $N \geq 0.$ We prove the following basic higher order estimate for the right hand side of our equation using the chain rule.

 \begin{lemma} \label{lem:estimateComposition}
  We have
 \begin{align*}
   \| \rho^{|\alpha| }\partial ^\alpha f(x,u,Du,D^2 u)\|_{\tilde H^m } 
& \leq P^{|\alpha|} ( \{ E^{|\gamma|}\|D^{|\gamma|} f 
\|_{C^m (K)}\}, \{M_k\}_{k=0, \ldots, |\alpha| })
 \end{align*}
 where $E = (1 + \|\rho\|_{ \tilde H^m})^3.$
 \end{lemma}
 
 \begin{proof}
Applying Faa di Brunos formula (Proposition \ref{prop:FraaDiBruno}) to 
$f\circ g$ where
$$
 g(x) = (x,u(x), Du(x), D^2u(x))
$$
we get
 \begin{align*}
   \rho^{|\alpha| }\partial ^\alpha f(x,u,Du,D^2 u)\ & =  \rho^{|\alpha|}P^{\alpha}_{n,n^2 + 2n +1} 
(\{\partial^\gamma f \}, \{\partial^\gamma g \}) 
\end{align*}
where $P^{\alpha}_{n,n^2 + 2n +1} 
(\{\partial^\gamma f \}, \{\partial^\gamma g \}) $ is a linear combination with positive coefficients of terms of the form 
$$
  \partial^l_{x_{i_1}, \ldots,  x_{i_l}}f( g(x)) \, \partial^{\gamma_1}g_{i_1} 
\cdots  \partial^{\gamma_l}g_{i_l}  
 $$
 with $1 \leq l \leq |\alpha|$ and $\gamma_1 + \ldots + \gamma_l = \alpha.$Due to the special structure of $g$, we have $\partial^\gamma g_i = 1$ for 
$i=1, \ldots, n$ and $|\gamma|=1$ and $\partial^\gamma g_i = 0$ for $i=1, \ldots, n$ 
and $|\gamma|\geq 2$.
For such terms we get using that $H^m$ is a Banach algebra \eqref{eq:BanachAlgebra}
 \begin{align*}
  \| \rho^{|\alpha|} \partial^l_{x_{i_1}, x_{i_l}} f( g(x))  
\partial^{\gamma_1}g_{i_1} \cdots  \partial^{\gamma_k}g_{i_l}\|_{\tilde H^m (B_1(0))}
  & \leq  E^l \| \partial^k_{x_{i_1}, x_{i_l}} f( g(x)) \|_{\tilde H^m (B_1(0))}
 M_{|\gamma_1|} \cdots M_{|\gamma_l|} \\
 & \leq  E^l \| \partial^l_{x_{i_1}, x_{i_l}} 
f( g(x)) \|_{\tilde H^m (B_1(0)}
 M_{|\gamma_1|} \cdots M_{|\gamma_l|}
 \end{align*}
 where $E = (\|\rho\|_{\tilde H^m} +1)^3.$
 We hence deduce using Lemma \ref{lem:SingleVariable} that  
  \begin{align*}
   \| \rho^{|\alpha| }\partial ^\alpha f(x,u,Du,D^2 u)\|_{\tilde H^m } 
& \leq C P^{|\alpha|} ( \{ E^{|\gamma|}\|D^{|\gamma|} f 
\|_{C^m (K)}\}, \{M_k\}_{k=0, \ldots, |\alpha| })
 \end{align*}
 \end{proof}

\section{The proof} \label{sec:proof}

After the above preliminaries, we can now prove Theorem \ref{thm:Analyticity}.

\subsection{Reduction of the problem}
Let $x_0 \in \Omega$. We will show that $u$ is analytic in a neighborhood of $x_0$ in $\Omega$. We can assume that $x_0 =0$, $B_1(0) \subset \Omega$, $a_ij (0) = \delta_{ij}$, and that the $a_{ij}$ have an arbitrarily small oscillation  on $B_1(0)$, say  smaller than $\varepsilon$ for an $\varepsilon >0$ to be chosen below. This can be achieved by a standard change of coordinates  combined with a scaling of the solution.
 
Differentiating Equation \eqref{eq:PDE}, we see that the function $v_k= \partial_k u$  satisfies the equation
\begin{align*}
 a_{ij} (x,u, D u, D^2 u) \partial_{ij} v_k (x) &=- ( b_i (x,u, Du, D^2u) \partial_i v_k(x)  + c (x,u, \Delta u, \Delta^2 u) v_k + d_k (x,u,Du, D^2 u) ) \\
 &=  f_k(x,u,D u, D^2u),
\end{align*}
where
\begin{align*}
 &a_{ij} (x,y, p,q) = \frac {\partial \phi }{\partial q_{ij}} (x,y, p,q), &b_i  (x,y, p,q) = \frac {\partial \phi}{\partial p_{i}} (x,y, p,q), \\
 &c(x,y, p,q) = \frac {\partial \phi}{\partial y} (x,y, p,q),  &d_k (x,y, p,q) = \frac  {\partial \phi}{\partial x_{k}}  (x,y,p,q).
\end{align*}
We can read off from its definition that $f_k$ is an analytic function around $0$. 

We now freeze the coefficients, i.e. we set  $a_{ij}^0 := a_{ij} (0,u(0),Du(0), D^2u(0))$ and observe that $v_k$ solves the equation
$$
 a_{ij} ^0 \partial_{ij} v_k = \tilde a_{ij}(x,u,Du,D^2 u) \partial_{ij} v_k + f_k(x,u,Du,D^2u)
$$
where $\tilde a_{ij} = a_{ij}^0 - a_{ij} $.  Due to the bound on the oscillation of the $a_{ij}$, we have
$$
 |\tilde a_{ij}| \leq \varepsilon 
$$ 
on $B_1(0).$

As
$
 a_{ij}^0 = \delta_{ij},
$
the function $v_k$ solves
\begin{equation} \label{eq:FrozenConstants}
 \Delta v_k = \tilde a_{ij} (x,u,Du, D^2 u) \partial_{ij} v_k + f_k (x,u,Du,D^2 u)= R_k.
\end{equation}

\subsection{A recursive inequality}
 
Since  $\tilde a_{ij} $ and $f_k$ are analytic, there are constants $C,A < \infty$ such that
for every multi-index $\alpha \in \mathbb N^{2n \times n^2}$ we have
\begin{equation} \label{eq:ControlFA}
 \|\partial^\alpha a_{ij}(x,u,Du,D^2u)\|_{C^m (B_1(0))}, \|\partial^{\alpha} 
f_k(x,u,Du,D^2u)  \|_{C^m (B_1(0))} \leq C A^{|\alpha|} |\alpha|!.
\end{equation}

For a fixed cutoff function $\rho \in C^\infty ( \mathbb R^n),  [0,1])$ with support 
contained in $B_1(0)$ and $\rho=1$ on $B_{\frac 12 }(0)$  we now consider as above the 
quantities
\begin{align*}
 \tilde M_N := \sup_{|\alpha |=N} \| \partial^\alpha u\|_{\tilde H^m} \text{ for } N = 0,1,2\\
 \tilde M_N := \sup_{|\alpha |=N} \| \rho^{N-3} \partial^\alpha u\|_{\tilde H^m} \quad \text{ for } N \geq 3
\end{align*}
and
\begin{align*}
 M_N := \tilde M_{N+2} + \tilde M_{N+1} + \tilde M_N +1
\end{align*}
for all $N \geq 0.$ We prove the following recursive estimate:

\begin{proposition} \label{prop:RecursiveEstimate}
For all $N \geq 1$
\begin{multline*}
 M_{N+1}\leq C \varepsilon M_{N+1}  +C  \sum_{0 < l \leq N}  
\binom{N}{l}P^{l}(\{E^kA^kk!\}, \{M_k\}) M_{(N + 1) -l} 
 \\
 +C P^{N}(\{E^k A^kk!\}, \{M_{k} \}) + C N M_N  + C N (N-1) M_{N-1}.
 \end{multline*}
\end{proposition}

\begin{proof}
We consider the weighted function $ \rho^{|\alpha|} (x) \partial^{\alpha +\beta}
v_k(x)$ for a multiindices $\alpha$ and $\beta$ with 
$|\beta|=2$ and $N = |\alpha| $.
Proposition \ref{prop:HmEstimates} and Equation \eqref{eq:FrozenConstants} yield
\begin{align*}
\| \rho^{N} \partial^{\alpha+\beta} v_k \|_{\tilde H^m} &\leq  \|\partial^\beta \rho^{N} \partial^{\alpha} v_k\|_{\tilde H^m} + \|[\rho^{N}, \partial^{\beta}] \partial^{\alpha} v_k\|_{\tilde H^m} \\
&\leq C \|\Delta \rho^{N} \partial^{\alpha} v_k\|_{\tilde H^m} + \|[\rho^{N}, \partial^{\beta}] \partial^{\alpha} v_k\|_{\tilde H^m} \\
&\leq C \| \rho^{N} \partial^{\alpha} \Delta v_k\|_{\tilde H^m}  + C \|[\Delta, \rho^{N}] \partial^{\alpha} v_k\|_{\tilde H^m} +\|[\rho^{N}, \partial^{\beta}] \partial^{\alpha} v_k\|_{\tilde H^m} \\
&= C \| \rho^{N} \partial^{\alpha} R_k\|_{\tilde H^m}  + C \|[\Delta, \rho^{N}] \partial^{\alpha} v_k\|_{\tilde H^m} +\|[\rho^{N}, \partial^{\beta}] \partial^{\alpha} v_k\|_{\tilde H^m} \\
&= I_1 + I_2 + I_3.
\end{align*}
Here, $[A,B]=AB - BA$ denotes the commutator. Using the identity
$$
 \partial_{ij} (\rho^{N} h) = \left\{ N (N-1) \rho^{N-2} \partial_i \rho \partial_j \rho +  N \rho^{N-1} \partial_{ij} \rho  \right\} h + N \rho^{N-1} \left(\partial_i \rho \partial_j h + \partial_j \rho \partial_i h \right ) + \rho^{N} \partial_{ij} h
$$
together with the fact that $H^m$ is a Banach algebra, we get
$$
 I_2+I_3 \leq C ( N M_N + N(N-1)M_{N-1}).
$$
To estimate the term $I_1$ we use 
\begin{align*}
 \|\rho^{N} \tilde a_{ij} \partial^\alpha \partial_{ij} \partial^\alpha v_k\|_{ \tilde H^m}
 \leq \|\tilde a_{ij}\|_{L^\infty} \|\rho^{N} \partial^\alpha v_k\|_{\tilde H^m}  + \|a_{ij}\|_{W^{m, \infty}} \|\rho^{N} \partial^\alpha v_k \|_{\tilde H^{m-1}}
 \leq \varepsilon M_{N+1} + C M_N.
\end{align*}
The product rule and the higher order chain rule together with the fact that $H^m$ is a Banach algebra yields
\begin{align*}
 \|\rho^{N} \partial^\alpha (\tilde a_{ij} \partial_{ij} v_k\|_{\tilde H^m} &\leq \|\rho^{N} \tilde a_{ij} \partial^{\alpha} \partial_{ij} v_k\|_{\tilde H^m} +\sum_{0 < \gamma \leq \alpha}  \binom{\alpha}{\gamma} \|\rho^{|\gamma|} \partial^{\gamma} \tilde a_{ij} \|_{\tilde H^m} \| \rho^{|\alpha| - |\gamma|}\partial^{\alpha-\gamma} \partial_{ij} v_k\| _{\tilde H^m}
\\
& \leq C \varepsilon M_{N+1} + C M_N +C  \sum_{0 < \gamma \leq \alpha}  
\binom{\alpha}{\gamma} P^{|\gamma|}(\{ E^k A^k k!\}, \{M_k\}) M_{|\alpha| - |\gamma| +1}
\end{align*}
where  we used Lemma \ref{lem:estimateComposition} and \eqref{eq:ControlFA} in the last step. Using Proposition \ref{prop:Binomials} we 
deduce
\begin{equation}
 \|\rho^{N+1} \partial^\alpha (\tilde a_{ij} \partial_{ij} v_k\|_{\tilde H^m} 
 \leq 
  C \varepsilon M_{N+1} + C M_N +C  \sum_{0 < l \leq N}  
\binom{N}{l} P^{l}(\{ E^k A^k k!\}, \{M_k\}) M_{(N +
1)-l}.
\end{equation}
Estimating the term $\|f_k\|_{\tilde H^m}$ in a similar way yields
$$
 I_1 \leq C \varepsilon M_{N+1}  + C M_N + C \sum_{0 <l  \leq N}  
\binom{N}{l}P^{l}(\{E^k A^k k!\}, \{M_k\}) M_{(N +1) - l} 
 + C P^{N}( \{E^k A^k k ! \}, \{M_{k}\})
$$
and hence finally
\begin{multline*} 
 M_{N+1}\leq C \varepsilon M_{N+1}  +C  \sum_{0 < l \leq N}  
\binom{N}{l}P^{l}(\{E^kA^kk!\},\{ M_k\}) M_{(N + 1) -l} 
 \\
 +C P^{N}(\{E^k A^k k!\}, M_{k}) + C N M_N  + C N (N-1) M_{N-1}.
\end{multline*}
\end{proof}

\subsection{Cauchy's method of majorants}
From Propositon \ref{prop:RecursiveEstimate} we get with $A_1= E \cdot A$
\begin{multline*} 
 M_{N+1}\leq C \varepsilon M_{N+1}  +C  \sum_{0 < l \leq N}  
\binom{N}{l}P^{l}(\{A_1^kk!\}, \{M_k\}) M_{(N + 1) -l} 
 \\
 +C P^{N}(\{A_1^kk!\}, M_{k}) + C N M_N  + C N (N-1) M_{N-1}
\end{multline*}
for all $N \in \mathbb N.$ We assume from now on that $C \geq 1$ and $C \varepsilon \leq \frac 1 2$, so that 
\begin{multline} \label{eq:recursiveInequality}
 M_{N+1}\leq \frac 1 2 M_{N+1}  +C  \sum_{0 < l \leq N}  
\binom{N}{l}P^{l}(\{A_1^kk!\}, \{M_k\}) M_{(N + 1) -l} 
 \\
 +C P^{N}(\{A_1^kk!\}, M_{k}) + C N M_N  + C N (N-1) M_{N-1}.
\end{multline}

We now look for a suitable majorant for the right-hand side. For $A_1 = E\cdot 
A$ we set
$$
  G_1(z) := \frac 1 2 + C \sum_{n \in \mathbb N} A_1^n (z-M_0)^n
$$
and
$$
 G_2(z) := C \sum_{n \in \mathbb N_0} A_1^n (z-M_0)^n .
$$
We consider the solution $w$ to ordinary the differential equation
$$
 w'(t) =  G_1(w) w' +  G_2(w) + C  (t+t^2) )w' + M_1$$
with $w(0) = M_0.$  As $|G_1(w)+ C(t+t^2)| <1 $ near to $t=0$ and $w= M_0$ this differential 
equations is equivalent to 
$$
 w'(t) = \frac {G_2(w)}{1-(G_1(w) + C t+ Ct^2)}
$$
near to $t_0$ and $w=M_0$. Hence, this initial value problem has a unique analytic 
solution on an open time interval containing $0$. If we assume that $C A_1 \geq M_1/2$, we get that
$$
 w'(0) \geq M_1.
$$
Using the product and chain rule and using that $G_1'(M_0) =0$, one sees that the derivatives $\hat M_N = 
w^{(N)}(0)$ satisfy $M_1 \leq \hat M_1$ and the equality
\begin{multline*}
\hat M_{N+1} = \frac 12 \hat M_{N+1}  + C \sum_{0 < l \leq N}  
\binom{N}{l}P^{l}(\{A_1^kk!\}, \{\hat M_k\}) \hat M_{N + 1 -l} 
 \\ + C P^{N}(\{C A_1^kk!\},\{ \hat M_{k}\}) + C N \hat M_N  + C N (N-1) \hat M_{N-1}
\end{multline*}
for $N\geq 1$. Comparing the last equality with  inequality \eqref{eq:recursiveInequality} we inductively deduce that
$$
 M_N \leq \hat M_N
$$
for all $N \in \mathbb N_0$.
Since $w$ is analytic around $0$, Lemma \ref{lem:CharacterizationAnalyticity} gives us constants $0< \tilde C , B $ such that
$\hat M_N = w^{(N)}(0) \leq \tilde C B^N N!$ for all $N \in \mathbb N_0$. Hence,
$$
 M_N \leq \hat M_N \leq \tilde C B^N N!
$$
and Lemma \ref{lem:CharacterizationAnalyticity} tells us  that $u$ is analytic on $B_1(0)$.

%\bibliographystyle{plain}
%\bibliography{Master}
\end{document}